\documentclass[graybox]{svmult}

\usepackage{type1cm}
\usepackage{makeidx}
\usepackage{graphicx}
\usepackage{multicol}
\usepackage[bottom]{footmisc}
\usepackage{newtxtext}
\usepackage{newtxmath}
\usepackage{tikz}
\usetikzlibrary{calc}
\usetikzlibrary{matrix}
\usepackage{caption}

\def\P{\mathrm{P}}
\def\BDM{\mathrm{BDM}}

\def\BDFM{\mathrm{BDFM}}
\def\RT{\mathrm{RT}}
\def\tnorm{|\!|\!|}

\makeatletter
\newcommand{\setfloattype}[1]{\renewcommand{\@captype}{#1}}
\makeatother

\begin{document}

\title*{A second order finite element method with mass lumping for wave equations in $H(\mathrm{div})$}
\author{Herbert Egger and Bogdan Radu}
\institute{Herbert Egger \at TU Darmstadt, \email{egger@mathematik.tu-darmstadt.de}
\and Bogdan Radu \at TU Darmstadt, \email{bradu@mathematik.tu-darmstadt.de}}

\maketitle

\abstract{We consider the efficient numerical approximation of acoustic wave propagation in time domain by a finite element method with mass lumping. 
In the presence of internal damping, the problem can be reduced to a second order formulation in time for the velocity field alone. For the spatial approximation we consider $H(\mathrm{div})$--conforming finite elements of second order. In order to allow for an efficient time integration, we propose a mass-lumping strategy based on approximation of the $L^2$-scalar product by inexact numerical integration which leads to a block-diagonal mass matrix. A careful error analysis allows to show that second order accuracy is not reduced by the quadrature errors which is illustrated also by numerical tests.}

\section{Motivation} \label{sec:1}

The propagation of acoustic sound in channels or ducts with a small extension in one of the spatial directions is substantially damped by friction at the walls. Averaging over the small direction then leads to systems with internal damping of the form 
\begin{alignat}{2}
  \partial_{t} u + \nabla p        &= -du \qquad \label{sys1}\\ 
  \partial_{t} p + \mathrm{div}\,u &= 0   \qquad \label{sys2}
\end{alignat}
with appropriate initial and boundary conditions. 
The variables $u$ and $p$ here denote the velocity and pressure fields, respectively, and for ease of notation, the equations are written in dimensionless form. 
The right hand side in \eqref{sys1} models the drag forces and $d$ denotes the corresponding dimensionless damping or drag coefficient. 

In the absence of damping, i.e., when $d=0$, the system \eqref{sys1}-\eqref{sys2} can be reduced to the second order wave equation for the pressure
\begin{align}
	\partial_{tt} p - \Delta p=0 \label{sys3}
\end{align}
which results from differentiating \eqref{sys2} and eliminating $u$ via equation \eqref{sys1}.
The efficient discretization of \eqref{sys3} can be obtained in various ways, e.g., by finite difference or finite element methods. The latter are more flexible concerning high-order approximations and the treatment of non-trivial domains but suffer from non-diagonal mass-matrices which hinder the efficient time-integration. This can be overcome by mass-lumping; we refer to \cite{Cohen02} for an overview about various methods and to \cite{Geevers18,Mulder01} for some particular results concerning mass-lumping for finite element approximations.

In the presence of damping, i.e., if $d\neq 0$, the elimination of the velocity $u$ from \eqref{sys1}--\eqref{sys2} leads to an integro-differential equation for the pressure whose time-integration is again non-trivial. Elimination of the pressure, on the other hand, again leads to a second order differential equation 
\begin{align}
  \partial_{tt} u + d\partial_{t} u- \nabla\mathrm{div}\,u = 0 \label{sys4}
\end{align}
but now for the vector valued velocity field $u$. 
The stable discretization of \eqref{sys4} by finite elements requires the use of $H(\mathrm{div})$--conforming spaces and novel mass lumping techniques are required for the efficient time integration. We refer to \cite{Cohen02} for corresponding results for $H(\mathrm{curl})$--conforming finite-elements  required in the context of electromagnetic wave propagation.

In a recent work \cite{EggerRadu18}, we considered the lowest-order discretization of the system \eqref{sys1}-\eqref{sys2} by $\BDM_1$--$\P_0$ finite-elements with mass-lumping as suggested by Wheeler and Yotov \cite{WheelerYotov06} in the context of porous medium flow. The resulting scheme is convergent of first order and super-convergence for the projected pressure can be utilized to obtain second-order convergence for the velocity by a non-local post-processing strategy. 
In this paper, we choose finite elements with better approximation properties which lead to second order approximations in the energy norm
\begin{align} \label{eq:ee}
  \|\partial_{t}u(t)-\partial_{t}u_h(t)\|_{L^2} + \|\mathrm{div}(u(t) - u_h(t))\|_{L^2}\leq C(u) h^2
\end{align}
without the need for post-processing. A novel mass-lumping strategy is proposed to allow for the efficient time integration and a careful analysis of the quadrature error is presented in order to establish the order optimal convergence rates \eqref{eq:ee}.
We here consider only approximations of second order on hybrid meshes in two space dimensions. The basic arguments of our analysis however can be used to investigate approximations of higher order and in three space dimensions.

The remainder of this note is organized as follows: 
In Section~\ref{sec:2}, we formally state our model problem and basic assumptions and then introduce its finite element approximation. In Section~\ref{sec:3}, we present some auxiliary estimates and then formulate and prove our main result in Section~\ref{sec:4}. Details about the numerical implementation are given in Section~\ref{sec:5} and for illustration, we present in Section~\ref{sec:6} some preliminary numerical tests.

\newpage 

\section{Problem statement and finite element approximation} \label{sec:2}

Throughout the presentation, we denote by $\Omega\subseteq\mathbb{R}^2$ a bounded polygonal Lipschitz domain and by $T>0$ a finite time horizon. 
We consider the system
\begin{alignat}{2}
  \partial_{tt} u +d\partial_{t} u- \nabla\mathrm{div}\,u &= 0, \qquad &&\text{in }\Omega  \label{s1}\\ 
  n\cdot u &= 0,\qquad &&\text{on }\partial\Omega. \label{s2}
\end{alignat}
The existence of a unique solution $u$ for \eqref{s1}--\eqref{s2} with 
given initial values $u(0)=u_0$ and $\partial_{t} u(0)=u_1$ 
can be established by semigroup theory; see \cite{EggerRadu18} for details. 
Moreover, any classical solution of \eqref{s1}--\eqref{s2} satisfies the variational identity
\begin{align}
  (\partial_{tt} u(t),v) +(d\partial_{t} u(t),v) +(\mathrm{div}\,u(t),\mathrm{div}\,v) = 0, \label{sys5}
\end{align}
for all $v\in H_0(\mathrm{div},\Omega)= \{u\in L^2(\Omega)^2\,:\, \mathrm{div}\,u\in L^2(\Omega)\text{ and } n\cdot u=0\text{ on }\partial\Omega\}$.
Here and below, we use $(\cdot,\cdot)$ to denote the standard $L^2$-scalar product. 

Let $T_h$ = $\{K\}$ be a quasi-uniform mesh of $\Omega$ comprised of triangles and parallelograms and $h$ be the mesh size.
We consider local approximation spaces
\begin{alignat}{2}
V(K) = \left\{\begin{array}{ll}
\RT_1(K) , & \text{ if $K$ is a triangle,}	\vspace{0.3em}\\
\BDFM_2(K) , & \text{ if $K$ is a parallelogram,}
\end{array}\right. \label{sys8}
\end{alignat}
with vector valued polynomial spaces $\RT_1(K)$ and $\BDFM_2(K)$ as defined in \cite{BoffiBrezziFortin13}; compare with Figure~\ref{fig:elements}. 
The global approximation spaces is then defined as
\begin{align*}
V_h = \{v_h\in H_0(\mathrm{div},\Omega)\,:\,\, v_h|_K\in V(K)\}. 
\end{align*}
The scalar product on $V_h$ will be approximated by $(u,v) \simeq (u,v)_{h}:=\sum_{K}(u,v)_{h,K}$ with local contributions obtained by numerical integration according to
\begin{align}
(u,v)_{h,K} = |K| \left(\alpha (u(m_K) \cdot v(m_K))+ \sum\nolimits_i\;\beta\,(u(v_{K,i})\cdot v(v_{K,i}))\right)
\label{sys9}
\end{align}
Here $m_K$ and $v_{K,i}$ represent the midpoint and vertices of the element $K$, respectively, while $\alpha$ and $\beta$ are the corresponding weights. 
\begin{figure}[ht!]
  \centering
  \begin{tikzpicture}[scale=0.42]
    \draw (0,0) -- (3,0);
    \draw (0,0) -- (0,3);
    \draw (3,0) -- (0,3);
    \draw[thick,->] (0,3) -- (0.5,3.5);
    \draw[thick,->] (3,0) -- (3.5,0.5);
    \draw[thick,->] (0,0) -- (-0.70,0);
    \draw[thick,->] (0,3) -- (-0.70,3);
    \draw[thick,->] (0,0) -- (0,-0.70);
    \draw[thick,->] (3,0) -- (3,-0.70);
    \draw[thick,->] (1,1) -- (1,1.70);
    \draw[thick,->] (1,1) -- (1.70,1);
    \draw[fill,red] (0,0) circle (0.12cm);
    \draw[fill,red] (3,0) circle (0.12cm);
    \draw[fill,red] (0,3) circle (0.12cm);
    \draw[fill,red] (1,1) circle (0.12cm);
  \end{tikzpicture}
  \qquad\qquad
  \centering
  \begin{tikzpicture}[scale=0.42]
    \draw (0,0) -- (3,0);
    \draw (0,0) -- (0,3);
    \draw (3,0) -- (3,3);
    \draw (0,3) -- (3,3);
    \draw[thick,->] (0,0) -- (-0.7,0);
    \draw[thick,->] (0,0) -- (0,-0.7);
    \draw[thick,->] (0,3) -- (0,3.7);
    \draw[thick,->] (0,3) -- (-0.7,3);
    \draw[thick,->] (3,0) -- (3.7,0);
    \draw[thick,->] (3,0) -- (3,-0.70);
    \draw[thick,->] (3,3) -- (3.7,3);
    \draw[thick,->] (3,3) -- (3,3.7);
    \draw[thick,->] (1.5,1.5) -- (1.5,2.2);
    \draw[thick,->] (1.5,1.5) -- (2.2,1.5);
    \draw[fill,red] (0,0) circle (0.12cm);
    \draw[fill,red] (3,0) circle (0.12cm);
    \draw[fill,red] (0,3) circle (0.12cm);
    \draw[fill,red] (3,3) circle (0.12cm);
    \draw[fill,red] (1.5,1.5) circle (0.12cm);
  \end{tikzpicture}
  \caption{\footnotesize Degrees of freedom for $\RT_1$ (left) and $\BDFM_2$ (right) and quadrature points (red dots).} \label{fig:elements}
\end{figure}
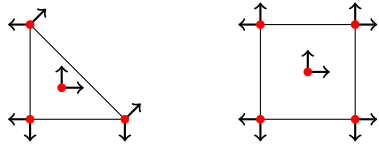%
On triangles, we choose $\alpha = \tfrac{3}{4}$ and $\beta = \tfrac{1}{12}$, while on parallelograms, we choose $\alpha = \tfrac{2}{3}$ and $\beta = \tfrac{1}{12}$. 
For the space discretization of \eqref{sys5}, we then consider the following inexact Galerkin scheme.
\begin{problem}\label{p:2}
Let $u_{h,0}$, $u_{h,1} \in V_h$ be given. 
Find $u_h:[0,T]\rightarrow V_h$ such that
\begin{align}
(\partial_{tt} u_h(t),v_h)_h + (d\partial_{t} u_h,v_h)_h + (\mathrm{div}\,u_h(t),\mathrm{div}\,v_h) = 0  \label{sys6}
\end{align}
for all $v_h\in V_h$ and all $t\in[0,T]$ and such that $u_h(0)=u_{h,0}$ and $\partial_{t} u_h(0)=u_{h,1}$.
\end{problem}
The following result ensures the well-posedness of Problem~\ref{p:2}.
\begin{lemma}
The inexact scalar product $(\cdot,\cdot)_h$ induces a norm on $V_h$ and, as a consequence, Problem~\ref{p:2} admits a unique solution. 
\end{lemma}
\begin{proof}
Choose any basis for $V_h$. Then the mass matrix associated with the inexact scalar product $(\cdot,\cdot)_h$ is symmetric and positive definite; this can be verified by elementary computations on single elements. Existence of a unique solution then follows from the Picard-Lindel\"of theorem.
\end{proof}

\vspace*{-2em}
\section{Auxiliary results} \label{sec:3}

In the following, we recall some well-known interpolation results 
and then derive estimates for the quadrature error which will be required below. 
Let us start with introducing a canonical interpolation operator which is defined locally by 
\begin{alignat}{2}
(\Pi_h u)|_K = \left\{\begin{array}{ll}
\Pi^{\RT}_K u|_K , & \text{ if $K$ is a triangle,}	\vspace{0.3em}\\
\Pi^{\BDFM}_K u|_K , & \text{ if $K$ is a parallelogram.}
\end{array}\right. \label{sys8}
\end{alignat}
Here $\Pi_K^{\RT}$ and $\Pi_K^{\BDFM}$ denote the standard interpolation operators for the local finite element spaces $\RT_1(K)$ and $\BDFM_2(K)$, respectively; see \cite{BoffiBrezziFortin13} for details. 
The following assertions then follow from well-known results about the local operators. 
\begin{lemma}\label{lem:interp}
Let $K\in T_h$ and $\Pi_h$ be defined as in \eqref{sys8}. Then
\begin{align}
  \|u-\Pi_h u\|_{L^2(K)}&\leq Ch^2\|u\|_{H^2(K)},
\end{align}
for all $u \in H(\mathrm{div},\Omega) \cap H^2(T_h)^2$ with constant $C$  independent of $h$. Moreover
\begin{align}
  (\mathrm{div}(u-\Pi_h u),\mathrm{div}\,v_h)=0,\qquad\forall\,v_h\in V_h.
\end{align}
\end{lemma}
We will further require the following property of the spaces $\RT_1(K)$ on triangles. 
\begin{lemma}\label{lem:split}
Let $K$ be a triangle. Then there exists a unique splitting 
\begin{align}
  \RT_1(K)=P_1(K)^2\oplus B(K)
\end{align}
and $\dim(B(K))=\dim(\mathrm{div}(B(K)))$. Therefore, $\|\mathrm{div}(\cdot)\|_{L^2}$ defines a norm on $B(K)$ and $ \|\mathrm{div}\,v_h^B\|_{L^2(K)} \ge c \|\nabla v_h^B\|_{L^2(K)}$ for any $v_h^B \in B(K)$ with $c$ independent of $K$.
\end{lemma}
These assertions can be verified by a elementary computations on the reference element and a mapping argument. 
As a next step, we summarize some properties of the numerical integration underlying the definition \eqref{sys9} of the inexact scalar product.
\begin{lemma}\label{lem:ex}
The quadrature rule in \eqref{sys9} is exact for polynomials of degree $k\le 2$ on triangles and for polynomials of degree $k\le 3$ on parallelograms.
\end{lemma}
The validity of these claims can again be verified by elementary computations on reference elements. 
In the following, we will abbreviate the quadrature errors by
\begin{align} \label{sys10}
\sigma_K(u,v)\coloneqq (u,v)_{h,K}-(u,v)_K 
\quad\text{and}\quad
\sigma_h(u,v) = \sum\nolimits_{K\in T_h}\sigma_T(u,v) 
\end{align}
Moreover, we denote by $\pi_K^k : L^2(K) \to P_k(K)^2$ the local $L^2$-orthogonal projections and we use $\pi_h^k : L^2(\Omega) \to P_k(T_h)^2$ to denote the corresponding global projection.
\begin{lemma}\label{lem:quaderr}
Let $u \in L^2(\Omega)^2$ with $u|_K \in H^1(K)^2$ for all $K \in T_h$. 
Then 
\begin{alignat}{2}
|\sigma_K(\pi_h^1 u, v_h)|
\leq \left\{\begin{array}{lll}
C h^2\|u\|_{H^1(K)}\|\mathrm{div}\,v_h\|_{L^2(K)}  &\text{ if $K$ is a triangle,}	\vspace{0.1em}\\
0,  &\text{ if $K$ is a parallelogram,}
\end{array}\right.
\end{alignat}
for all $v_h\in V_h$ and all $K\in T_h$ with constant $C$ independent of the element $K$.
\end{lemma}
\begin{proof}
From Lemma~\ref{lem:ex}, we deduce that $|\sigma_K(\pi_h^1 u, v_h)|=0$ on parallelograms. 
For triangles, on the other hand, we can estimate the quadrature error by
\begin{align*}
|\sigma_K(&\pi_h^1 u, v_h)|
 =|\sigma_K(\pi_h^1u-\pi_h^0u,v_h-\pi_h^1v_h)|\\
&\le \|\pi_h^1u-\pi_h^0u\|_{L^2(K)}\|v_h-\pi_h^1v_h\|_{L^2(K)} + \|\pi_h^1u-\pi_h^0u\|_{h}\|v_h-\pi_h^1v_h\|_{h}\\
&\le Ch^3\|u\|_{H^1(K)}\|\nabla^2 v_h\|_{L^2(K)}.
\end{align*}
By Lemma~\ref{lem:split}, we can split $v_h = v_h^1 \oplus v_h^B$ on $K$ and further estimate
\begin{align*}
\|\nabla^2 v_h\|_{L^2(K)}
&= \|\nabla^2 v_h^B\|_{L^2(K)}
 \le C' h^{-1} \|\nabla v_h^B\|_{L^2(K)} 
 \le C'' h^{-1} \|\mathrm{div}v_h^B\|_{L^2(K)}.
\end{align*}
The linear independence of the splitting also yields $\|\mathrm{div}v_h^B\|_{L^2(K)} \le C \|\mathrm{div}v_h\|_{L^2(K)}$, 
and a combination of the estimates already yields the bound for the triangles. 
\end{proof}

\section{Convergence analysis} \label{sec:4}

For ease of notation, we will only consider the case $d=0$ in the sequel. 
As usual, we begin with splitting the error in interpolation and discrete error components by
\begin{align}
u-u_h = (u-\Pi_h u)+(\Pi_h u-u_h) =: -\eta + \psi_h.
\end{align}
The discrete error component can be estimated as follows.
\begin{lemma}\label{lem:discerr}
Let $u$ and $u_h$ denote the solutions of \eqref{sys5} and \eqref{sys6} with initial values linked by $u_h(0)= \Pi_h u(0)$ and $\partial_t u_h(0) = \Pi_h \partial_t u(0)$. Then the discrete error satisfies
\begin{align*}
\|\partial_{t}(\Pi_hu - u_h)\|_{L^\infty(0,T;L^2(\Omega))} + \|\mathrm{div}\,(\Pi_hu - u_h)\|_{L^\infty(0,T;L^2(\Omega))} \leq C_1(u,T)h^2
\end{align*}
with constant $C_1(u,T)=C_1'(u,T)+C_1''(u,T)$ as defined in the proof below.
\end{lemma}
\begin{proof}
The discrete error $\psi_h=\Pi_hu - u_h$ can be seen to satisfy the identities
\begin{align}
  (\partial_{tt}\psi_h(t),v_h)&+(\mathrm{div}\,\psi_h(t),\mathrm{div}\,v_h) = \\ &(\partial_{tt}\eta(t),v_h)+(\mathrm{div}\,\eta(t),\mathrm{div}\,v_h) +
  \sigma_h(\Pi_h\partial_{tt} u(t),v_h)
\end{align}
for all $v_h \in V_h$ and $0 \le t \le T$. Moreover,  $\psi_h(0)=\partial_{t} \psi_h(0)=0$ by construction. 
Choosing $v_h=\partial_t \psi_h(t)$ as a test function followed by integrating from $0$ to $t$ leads to
\begin{align} \label{eq:identity}
&\frac{1}{2}\left(\|\partial_{t}\psi_h(t)\|_h^2+\|\mathrm{div}\,\psi_h(t)\|_{L^2(\Omega)}^2\right)\\
&\quad =\smallint_0^t(\partial_{tt}\eta(s),\partial_{t}\psi_h(s)) + (\mathrm{div}\,\eta(s),\mathrm{div}\,\partial_{t}\psi_h(s)) + 
\sigma_h(\Pi_h\partial_{tt} u(s),\partial_{t}\psi_h(s)) \, ds \notag\\
&\quad =:(i)+(ii)+(iii). \notag
\end{align}
Using Cauchy-Schwarz and Young's inequalities, the first term can be estimated by
\begin{align}
(i) \leq C_1'(u)^2h^4 + \tfrac{1}{4}\|\partial_{t}\psi_h\|^2_{L^\infty(0,t,L^2(\Omega))}
\end{align}
with constant $C_1'(u,t)=C\|\partial_{tt} u\|_{L^1(0,t,H^2(\Omega))}$,
and by Lemma~\ref{lem:interp}, we get $(ii)=0$. 
The remaining third term can finally be estimated by
\begin{align*}
(iii)&=\smallint_0^t \sigma_h(\Pi_h\partial_{tt} u(s)-\pi_h^1\partial_{tt} u(s),\partial_{t}\psi_h(s)) +
\smallint_0^t \sigma_h(\pi_h^1\partial_{tt} u(s),\partial_{t}\psi_h(s))\\
&=:(iv)+(v).
\end{align*}
The term $(iv)$ can be bounded with the same arguments $(i)$. 
If $K$ is a parallelogram, then $(v)\equiv 0$ by Lemma~\ref{lem:quaderr}. 
On triangles, we use integration-by-parts in time, to get 
\begin{align*}
(v)&=\sigma_h(\pi_h^1\partial_{tt} u(t),\psi_h(t))-\smallint_0^t \sigma_h(\pi_h^1\partial_{ttt} u(s),\psi_h(s))\,ds\\
&\leq 
C_1''(u,t)^2h^4 + \tfrac{1}{2}\|\mathrm{div}\,\psi_h\|^2_{L^\infty(0,t,L^2(\Omega))}
\end{align*}
with $C_1''(u,t)=C(\|\partial_{tt} u\|_{L^\infty(0,t,H^1(\Omega))}+\|\partial_{ttt} u\|_{L^1(0,t,H^1(\Omega))})$, where we used Lemma~\ref{lem:quaderr} in the second step.
Taking the supremum over $t\in[0,T]$ in \eqref{eq:identity} and absorbing all the terms with the test function into the left side of \eqref{eq:identity} now yields the assertion.
\end{proof}
\begin{theorem}
Let the assumptions of Lemma~\ref{lem:discerr} hold. 
Then
\begin{align}
\|\partial_{t} (u-u_h)\|_{L^\infty(0,T;L^2(\Omega))} + \|\mathrm{div}(u-u_h)\|_{L^\infty(0,T;L^2(\Omega))} \leq C(u,T)h^2,
\end{align}
with constant $C(u,T)=C_1'(u,T)+C_1''(u,T)+C_2(u,T)$ as in the proof below.
\end{theorem}
\begin{proof}
Using Lemma~\ref{lem:interp}, we can estimate the interpolation error by
\begin{align}
\|\partial_{t} \eta\|_{L^\infty(0,T;L^2(\Omega))} + \|\mathrm{div}\,\eta\|_{L^\infty(0,T;L^2(\Omega))} \leq C_2(u,T)h^2,
\end{align}
with $C_1(u,t)=C(\|\partial_{t} u\|_{L^\infty(0,t;H^2(\Omega)}+\|\mathrm{div}\,u\|_{L^\infty(0,t;H^2(\Omega)})$. The proof is completed by adding the bounds for the discrete error components provided by Lemma~\ref{lem:discerr}.
\end{proof}

\section{Implementation and mass lumping}\label{sec:5}

For completeness, we now briefly introduce appropriate basis functions for the spaces $\RT_1(K)$ and $\BDFM_2(K)$ which together with the inexact scalar product $(\cdot,\cdot)_h$ lead to a block-diagonal mass matrix. 
Let $\{\lambda_i\}$ denote the barycentric coordinates of the element $K$ and let $\nabla^\perp f = (\partial_y f,-\partial_x f)^T$. 
On triangles, we define
\begin{align*}
  \Phi_{B1} = \lambda_2(\lambda_1\nabla^\perp \lambda_3-\lambda_3\nabla^\perp \lambda_1)
  \quad\text{ and }\quad
  \Phi_{B2} = \lambda_3(\lambda_1\nabla^\perp \lambda_2-\lambda_2\nabla^\perp \lambda_1)
\end{align*}
which are the two $H(\mathrm{div})$-bubble functions associated with the element midpoint; see Figure~\ref{fig:elements}. 
The basis functions associated with the three vertices are given by 
\begin{alignat*}{2}
&\Phi_{1,1} = \lambda_1\nabla^\perp \lambda_2 + \Phi_{B1}-2\Phi_{B2}, \qquad
&&\Phi_{1,2} = \lambda_2\nabla^\perp \lambda_1 + \Phi_{B1}+\Phi_{B2},\\
&\Phi_{2,1} = \lambda_2\nabla^\perp \lambda_3 -2\Phi_{B1}+\Phi_{B2}, \qquad
&&\Phi_{2,2} = \lambda_3\nabla^\perp \lambda_2 +\Phi_{B1}-2\Phi_{B2},\\
&\Phi_{3,1} = \lambda_1\nabla^\perp \lambda_3 -2\Phi_{B1}+\Phi_{B2}, \qquad
&&\Phi_{3,2} = \lambda_3\nabla^\perp \lambda_1 +\Phi_{B1}+\Phi_{B2}.
\end{alignat*}
For parallelograms, let $\xi_{ij} \in [0,1]$ denote the local coordinate on the edge $e_{ij}$ pointing from vertex $p_i$ to $p_j$. Following the construction in \cite{Zaglmayr06}, we define by
\begin{align*}
\phi_{B1} = (\lambda_1+\lambda_4)(\lambda_2+\lambda_3)\nabla^\perp\xi_{23} \quad\text{ and }\quad
\phi_{B1} = (\lambda_1+\lambda_2)(\lambda_3+\lambda_4)\nabla^\perp\xi_{12}
\end{align*}
two $H(\mathrm{div})$-bubble functions associated with the midpoint of the element. For any of the four vertices, we further define two basis functions by 
\begin{alignat*}{2}
&\phi_{1,1} = \lambda_2\nabla^\perp\xi_{23} + \phi_{B1},\qquad 
&&\phi_{1,2} = \lambda_3\nabla^\perp\xi_{23} + \phi_{B1},\\
&\phi_{2,1} = \lambda_3\nabla^\perp\xi_{34} + \phi_{B2},\qquad 
&&\phi_{2,2} = \lambda_4\nabla^\perp\xi_{34} + \phi_{B2},\\
&\phi_{3,1} = \lambda_4\nabla^\perp\xi_{41} - \phi_{B1},\qquad 
&&\phi_{3,2} = \lambda_1\nabla^\perp\xi_{41} - \phi_{B1},\\
&\phi_{4,1} = \lambda_1\nabla^\perp\xi_{12} - \phi_{B2},\qquad 
&&\phi_{4,2} = \lambda_2\nabla^\perp\xi_{12} - \phi_{B2}.
\end{alignat*}
Let us note that by construction, exactly two basis functions are associated to any of the quadrature points. Moreover, the basis functions vanish on all quadrature points except one. As a consequence, the local mass matrix corresponding to $(\cdot,\cdot)_{h,K}$ is block diagonal with $2 \times 2$ blocks. After assembling, the global mass-matrix is block-diagonal with each block corresponding to one of the quadrature points. The dimension of the individual blocks is determined by the number of degrees of freedom associated with that quadrature point; we refer to \cite{EggerRadu18,WheelerYotov06} for details.

\section{Numerical illustration} \label{sec:6}

For illustrating our resuts, we consider a simple test problem in two space dimensions, whose analytical solution is given by the plane wave
\begin{align*}
u_{ex}(x,y,t)=g(x-t) \binom{1}{0} 
\quad\text{with}\quad 
g(x)=2\exp(-50(x+1)^2). 
\end{align*}
We consider problem \eqref{sys4} with $d=0$ on the domain $\Omega =(0,1)^2$ with boundary and initial conditions obtained from the exact solution. 
In Table~\ref{tab:1}, we display the errors obtained by our second-order finite-element approximation with mass-lumping on a sequence of quasi-uniform but non-nested meshes with decreasing mesh size.
As predicted by our theoretical results, we observe second order convergence.
\begin{figure}[ht!]
  \centering
  \begin{minipage}{0.25\textwidth}\centering
  \setfloattype{table}
  \begin{tabular}{c||c|c}
    $h$ & $\tnorm \pi_h^1 u - u_h\tnorm$ & eoc\\
    \hline
    \hline
    \rule{0pt}{2.1ex}
    $2^{-3}$  & $0.270790$ & ---     \\
    $2^{-4}$  & $0.060266$ & $2.17$  \\
    $2^{-5}$  & $0.016328$ & $1.88$  \\
    $2^{-6}$  & $0.004343$ & $1.91$ \\
    $2^{-7}$  & $0.001046$ & $2.06$
  \end{tabular}
  \caption{\footnotesize Discrete error of the method. Time step was fixed at  $\tau=0.001$. \label{tab:1}}
  \end{minipage}
  \hfil
  \hfil
  \begin{minipage}{0.7\textwidth}\centering
    \centering
    \begin{center}
      \includegraphics[width=.30\textwidth]{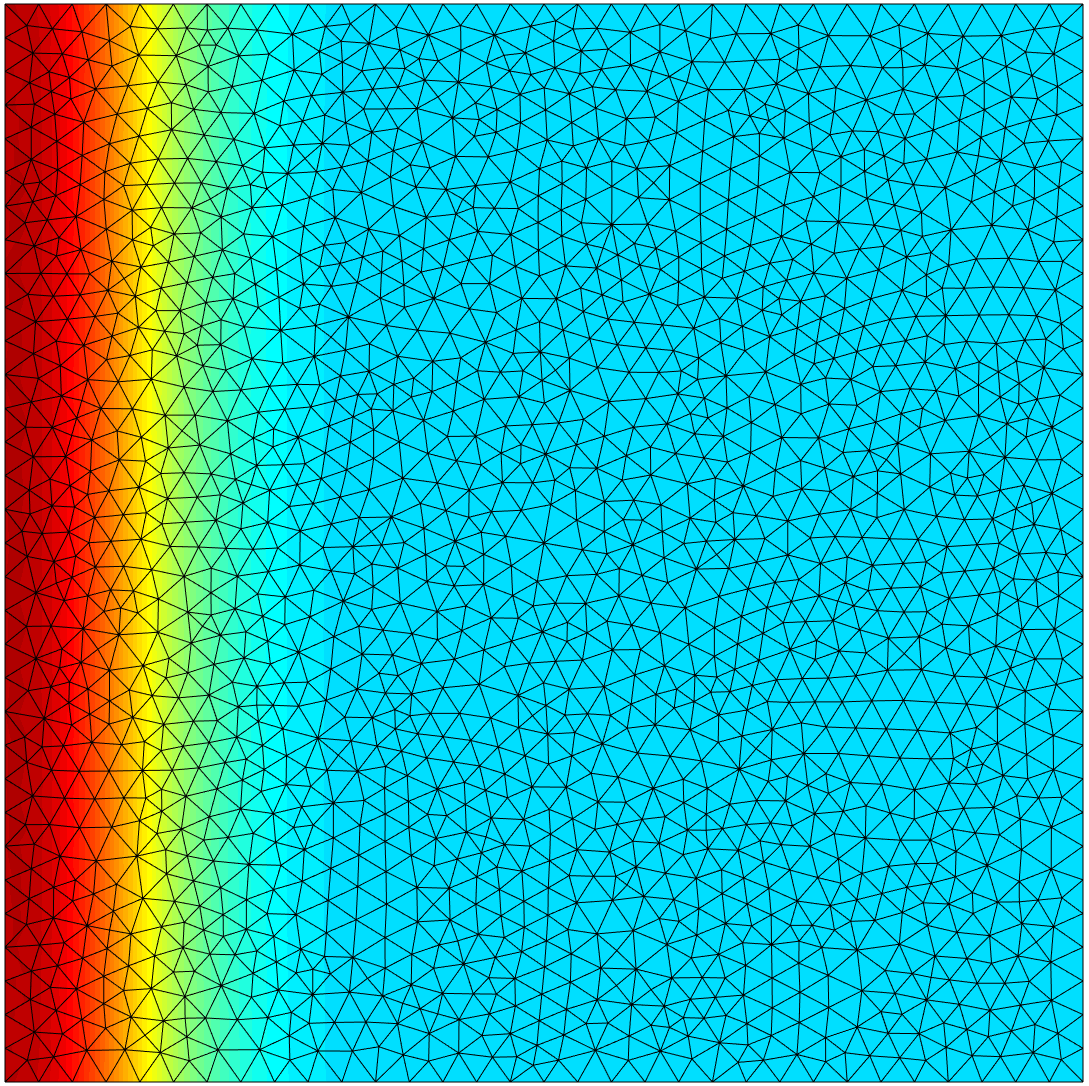}
      \hskip1ex
      \includegraphics[width=.30\textwidth]{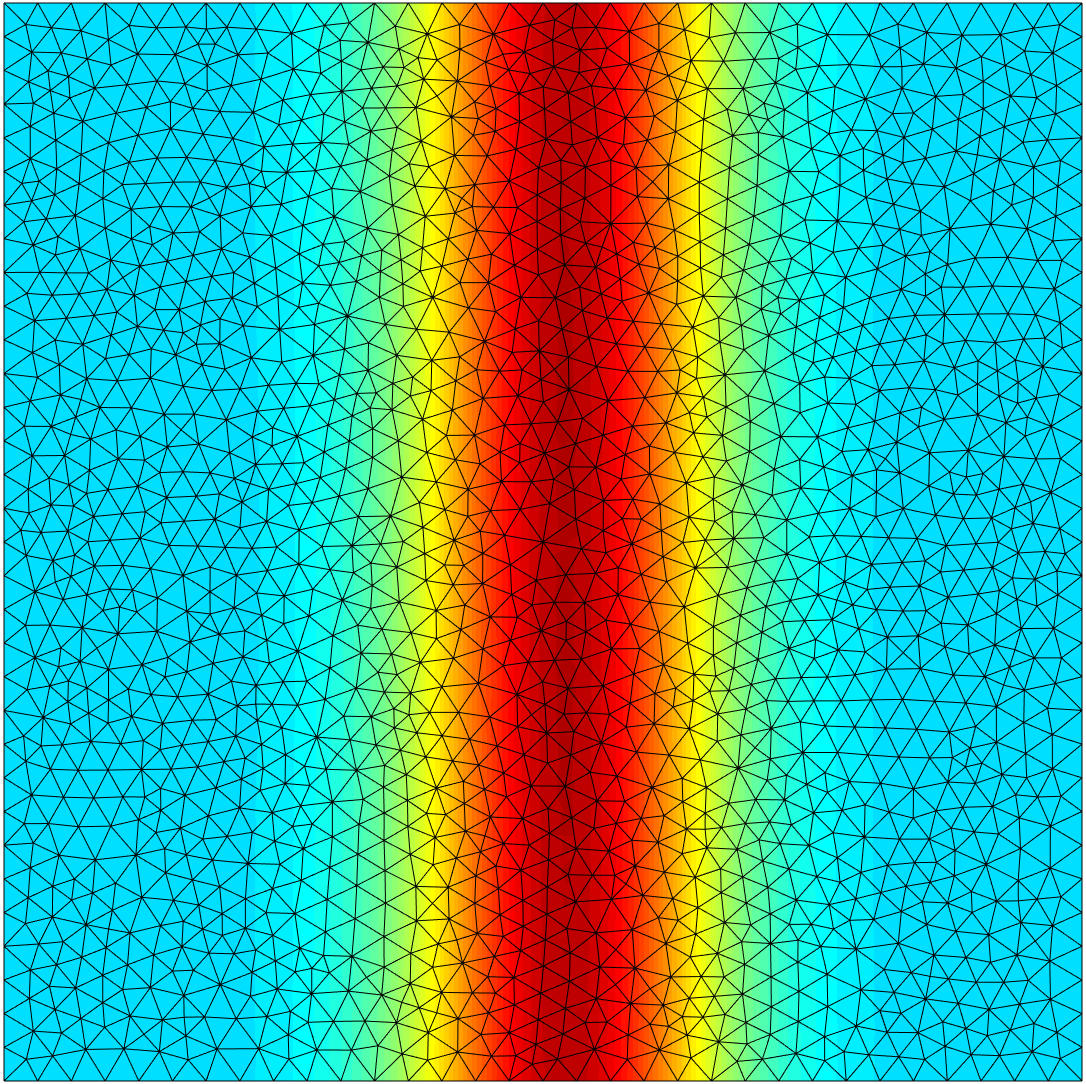}
      \hskip1ex
      \includegraphics[width=.30\textwidth]{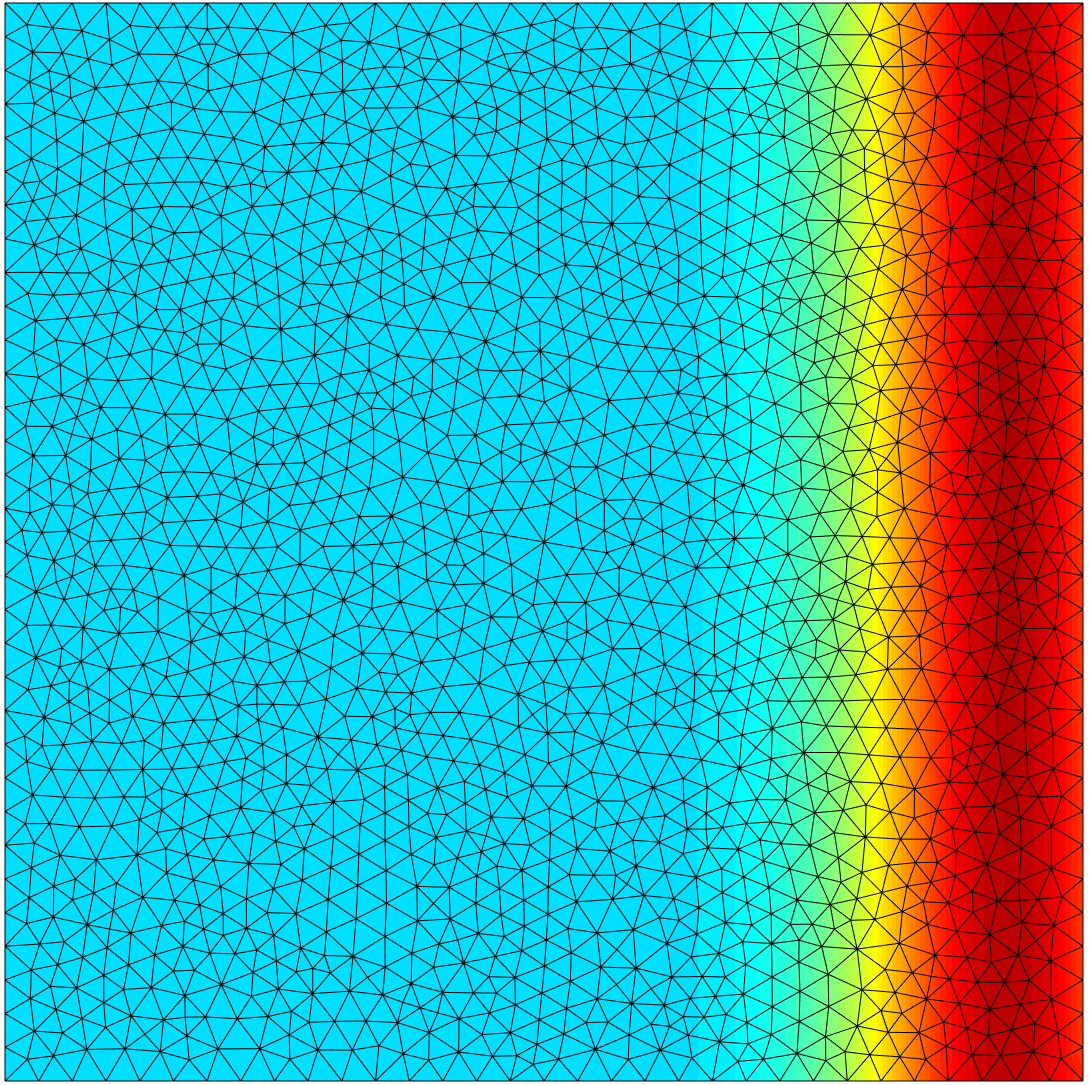} \\[1em]
      
      \captionsetup{justification=centering}
      \caption{\footnotesize Snapshots of the first component of $u$ at different time steps. \label{fig:pics}}
    \end{center}
  \end{minipage}
\end{figure}

Due to the mass lumping, time integration could be performed efficiently by the leapfrog scheme with time-step $\tau \approx h$. Since this method is second order accurate, this choice does not influence the overall convergence behavior; see \cite{Cohen02} for details. 

\begin{acknowledgement}
This work was supported by the German Research Foundation (DFG) via grants TRR~146 C3, TRR~154 C4, Eg-331/1-1, and through the ``Center for CE'' at TU Darmstadt.
\end{acknowledgement}

\end{document}